\documentclass[12pt]{amsart}
\usepackage{amsmath,amssymb,amsfonts,enumerate,amsthm}

\numberwithin{equation}{section}

\newtheorem{thm}{Theorem}[section]
\newtheorem{lem}[thm]{Lemma}
\newtheorem{cor}[thm]{Corollary}
\newtheorem{prop}[thm]{Proposition}
\newtheorem{rem}[thm]{Remark}
\newtheorem{defin}[thm]{Definition}
\newtheorem{defins}[thm]{Definitions}
\newtheorem{remi}[thm]{Reminder}
\newtheorem{nota}[thm]{Notation}

\def\sm{ \smallskip}

\def\e{ \begin{enumerate} \it }
\def\ee{\end{enumerate} }
\def\en{\operatorname{end}}
\def\beg{\operatorname{beg}}
\def\height{\operatorname{height}}
\def\depth{\operatorname{depth}}
\def\reg{\operatorname{reg}}

\def\Hom{\operatorname{Hom}}
\def\Spec{\operatorname{Spec}}
\def\Proj{\operatorname{Proj}}
\def\Var{\operatorname{Var}}

\def\beg{\operatorname{beg}}
\def\en{\operatorname{end}}

\def\Proj{\operatorname{Proj}}

\def\gendeg{\operatorname{gendeg}}
\def\length{\operatorname{length}}
\def\Ass{\operatorname{Ass}}
\def\dim{\operatorname{dim}}

\def\sk{\smallskip\par}

\newcommand{\p}{\mathfrak{p}}
\newcommand{\X}{\mathcal{P}}

\newcommand{\MS}{\mathbb{S}}
\newcommand{\Z}{\mathbb{Z}}

\newcommand{\N}{\mathbb{N}}

\newcommand{\D}{{\mathcal D}}

\newcommand{\C}{\mathcal{C}}

\begin{document}
\bibliographystyle{amsplain}

\author{Markus Brodmann}
\address{Markus Brodmann, University of Z\"urich, Institute of Mathematics, Winterthurerstrasse 190, 8057 Z\"urich. }
\email{brodmann@math.uzh.ch}

\author{Maryam Jahangiri}
\address{Maryam Jahangiri, School of Mathematics and Computer Sciences, Damghan University of Basic Sciences, Damghan,  Iran.}
\email{jahangiri@dubs.ac.ir}

\author{Cao Huy Linh}
\address{Cao Huy Linh, Department of Mathematics, College of Education, Hue University, 32 Le Loi, Hue City, Vietnam.}
\email{huylinh2002@yahoo.com}

\thanks{\today }
\subjclass[2000]{}

\title[Boundedness of Cohomology]
{Boundedness of Cohomology}

%---------------------------------------------------------------------------------------------------------------------------abstract
\begin{abstract}
Let $d \in \N$ and let $\D^d$ denote the class of all pairs
$(R,M)$ in which $R = \bigoplus_{n \in \N_0} R_n$ is a Noetherian
homogeneous ring with Artinian base ring $R_0$ and such that $M$
is a finitely generated graded $R$-module of dimension $\leq d$.

\sm

The cohomology table of a pair $(R,M) \in \D^d$ is defined as the
family of non-negative integers $d_M:= (d^i_M(n))_{(i,n) \in
\N \times \Z}$. We say that a subclass
$\mathcal{C}$ of $\D^d$ is of finite cohomology if the set $\{d_M
\mid \ (R,M) \in \C\}$ is finite. A set $\mathbb{S} \subseteq
\{0,\cdots ,d-1\}\times \Z$ is said to bound cohomology, if for
each family $(h^\sigma)_{\sigma \in \mathbb{S}}$ of non-negative
integers, the class $\{(R,M) \in \D^d\mid \ d^i_M(n) \leq h^{(i,n)}
\ \mbox{ for all} \ (i,n) \in \mathbb{S}\}$ is of finite
cohomology. Our main result says that this is the case if and only if
$\mathbb{S}$ contains a quasi diagonal, that is a set of the form
$\{(i,n_i)| \ i=0,\cdots ,d-1\}$ with integers $n_0> n_1 > \cdots
> n_{d-1}$.

\sm

We draw a number of conclusions of this boundedness criterion.
%\noindent
\end{abstract}
\maketitle
%-----------------------------------------------------------------------------------------------------------------------introduction
\section{Introduction} \sk
This paper continues our investigation [6], which was driven by
the question \it{"What bounds cohomology of a projective scheme?"}

\smallskip

{\rm A considerable number of contributions has been given to this
theme, mainly under the aspect of bounding some cohomological
invariants in term of other invariants (see [1], [2], [3], [4], [7],
 [8], [9], [11], [12], [13], [15], [16], [17], [18], [19], [21],
[22] for example).

\smallskip

Our aim is to start from a different point of view, focussing on
the notion of cohomological pattern (s. [5]). So, our main result
characterizes those sets $\mathbb{S}\subseteq \{0, \cdots,
d-1\}\times \mathbb{Z}$  "which bound cohomology of projective
schemes of dimension $< d$".

\smallskip

To make this precise, fix a positive integer $d$ and let ${\D}^d$
be the class of all pairs $(R, M)$ in which $R= \bigoplus_{n\geq
0}R_n$ is a Noetherian homogeneous ring with Artinian base ring
$R_0$ and $M$ is a finitely generated graded
 $R$-module with dim$(M)\leq d$. In this situation let $R_+= \bigoplus_{n> 0}R_n$ denote the irrelevant ideal of $R$.

\smallskip

For each $i\in {\N}_0$
consider the graded $R$-module $D^i_{R_+}(M)$, where $D^i_{R_+}$
denotes the $i$-th right derived functor of the $R_+$-transform
functor $D_{R_+}(\bullet):= \displaystyle \lim_{n
\longrightarrow\infty}\Hom_R((R_+)^n, \bullet )$. In addition, for each $n\in \Z$ let $d^i_M(n)$ denote
the (finite) $R_0$-length of the $n$-th graded component $D^i_{R_+}(M)_n$ of
$D^i_{R_+}(M)$.

\smallskip

Finally, for $(R, M)\in {\D}^d$ let us consider the so called
cohomology table of $(R, M)$, that is the family of non negative
integers
\[d_M:= (d^i_M(n))_{(i, n)\in \mathbb{N}_0 \times \mathbb{Z}}.\]
A subclass $\mathcal{C}\subseteq {\D}^d$ is said to be of finite
cohomology if the set $\{d_M\mid \ (R, M)\in {\C}\}$ is finite.
The class $\C$ is said to be of bounded cohomology if the set
$\{d^i_M(n)\mid \ (R, M)\in {\C}\}$ is finite for all pairs $(i,
n)\in \mathbb{N}_0 \times \mathbb{Z}$. It turns out that these two
conditions are booth equivalent to the condition that the class
$\C$ is of finite cohomology "along some diagonal", e.g. there is
some $n_0\in \mathbb{Z}$ such that the set $\triangle_{\C,
n_0}:=\{d^i_M(n_0-i)\mid \ (R, M)\in \C, 0\leq i<d \}$ is finite
(s. Theorem 3.5).

\smallskip

So, if one bounds the values of $d^i_M(n)$ along a "diagonal
subset" $$\{(j, n_0 -j)\mid \ j=0, \cdots, d-1\}\subseteq \{0,
\cdots, d-1 \}\times \Z$$ for an arbitrary integer $n_0$ one cuts
out a subclass $\mathcal{C}\subseteq \D^d$ of finite cohomology.
Motivated by this observation we say that the subset
$\MS\subseteq \{0, \cdots, d-1\}\times \Z $ bounds cohomology in
the class $\mathcal{C}\subseteq \D^d$ if for each family
$(h^{\sigma})_{\sigma \in \MS}$ of non-negative integers
$h^{\sigma}\in {\N}_0$ the class $$\{(R, M)\in \mathcal{C} \mid \
\forall (i, n)\in \MS: \ d^i_M(n)\leq h^{(i, n )}\}$$ is of finite
cohomology. Now, we may reformulate our previous result by saying
that for arbitrary $n_0$ the diagonal set $\{(j, n_0-j)\mid \
j=0, \cdots, d-1\}$ bounds cohomology  in $\D^d$. It seems rather
natural to ask, whether one can characterize the shape of those
subsets $\MS\subseteq \{0, \cdots, d-1\}\times \Z$ which bound
cohomology in $\D^d$. This is indeed done by our main result (s.
Corollary 4.10):

\sm

\renewcommand{\descriptionlabel}[1]%
             {\hspace{\labelsep}\textrm{#1}}
\begin{description}
\setlength{\labelwidth}{13mm} \setlength{\labelsep}{1.3mm}
\setlength{\itemindent}{0mm}
\item[ ]  \it{A subset $\MS\subseteq \{0, \cdots, d-1\}\times \Z$ bounds cohomology in $\D^d$ if
and only if it contains a quasi-diagonal, that is a set of the
form $\{(i, n_i)\mid \ i=0, \cdots, d-1\}$ with}
\[n_0> n_1> \cdots> n_{d-1}.\]
\end{description}

\sm

Our next aim is to apply the previous  result in order to cut out classes $\mathcal{C}\subseteq \D^d$ of
finite cohomology by fixing some numerical invariants which are
defined on the class $\mathcal{C}$. A finite family
$(\mu_i)_{i=1}^{r}$ of numerical invariants $\mu_i$ on
$\mathcal{C}$ is said to bound
cohomology in $\mathcal{C}$ if for all $n_1, \cdots, n_r\in
\Z\cup \{\pm\infty\}$ the class $\{(R, M)\in \mathcal{C}\mid \
\mu_i(M)=n_i \ \mbox{for} \ i=1, \cdots, r\}$ is of finite
cohomology.

\sm

We define a numerical invariant $\varrho: \D^d\rightarrow{\N}_0$
by setting $\varrho(M):= d^0_M(\reg^2(M))$, where $\reg^2(M)$
denotes the Castelnuovo-Mumford regularity of $M$ at and above
level 2. Then, we show (s. Theorem 5.8):}

\sm

\renewcommand{\descriptionlabel}[1]%
             {\hspace{\labelsep}\textrm{#1}}
\begin{description}
\setlength{\labelwidth}{13mm} \setlength{\labelsep}{1.3mm}
\setlength{\itemindent}{0mm}
\item[ ] {\it The pair of invariants $(\reg^2, \varrho)$ bounds
cohomology in $\D^d$.}
\end{description}

\sm

{\rm As an application of this  we prove (s. Theorem 5.9 and
Corollary 5.10)}

\sm

\renewcommand{\descriptionlabel}[1]%
             {\hspace{\labelsep}\textrm{#1}}
\begin{description}
\setlength{\labelwidth}{13mm} \setlength{\labelsep}{1.3mm}
\setlength{\itemindent}{0mm}
\item[ ] \it{ Fix a polynomial $p\in \mathbb{Q}[t]$ and an integer $r$.
Let $\mathcal{C}\subseteq \D^d $ be the class of all pairs
  $(R, M)$ such that $M$ is a graded submodule of a
finitely generated graded $R$-module $N$ with Hilbert polynomial $p_N= p$ and
$\reg^2(N)\leq r$. Then $\reg^2$ bounds
  cohomology in $\mathcal{C}$.}
\end{description}

\sm

{\rm An immediate consequence of this is (s. Corollary 5.11):}

\sm

\renewcommand{\descriptionlabel}[1]%
             {\hspace{\labelsep}\textrm{#1}}
\begin{description}
\setlength{\labelwidth}{13mm} \setlength{\labelsep}{1.3mm}
\setlength{\itemindent}{0mm}
\item[ ]  \it{Let $(R, N)\in \D^d$, let $r\in \Z$ and let $M$ run through
all graded submodules $M\subseteq N$
  with $\reg^2(M)\leq r$. Then only finitely many
cohomology tables $d_M$ occur.} \end{description}

\sm

{\rm As applications of this, we generalize two finiteness results
of Hoa-Hyry [17] for local cohomology modules of graded ideals in
a polynomial ring over a field to graded submodules $M\subseteq N$
for a given pair $(R, N)\in \D^d$ (s. Corollaries 5.13 and 5.14).

\sm

In order to translate our results to sheaf cohomology of
projective schemes observe that for all $(i, n)\in
\mathbb{N}_0\times \mathbb{Z}$ and all pairs $(R, M)\in {\D}^d$
we have $H^i(X, \mathcal{F}(n))\cong D^i_{R_+}(M)_n$, where $X:=
\Proj(R)$ and $\mathcal{F}:=\tilde{M}$ is the coherent sheaf of
$\mathcal{O}_X$-modules induced by $M$ (see [10, chap. 20] for
example).}

 \vskip 1 cm
%%%%%%%%%%%%%%%%%%%%%%%%
\section{ Preliminaries}%---------------------------------------------------------------------------------------------------------------------sec.2
%%%%%%%%%%%%%%%%%%%%%%%%

{\rm In this section we recall a few basic facts which shall be
used later in our paper.}

\begin{nota}\label{2.1}%----------------------------------------------------------------------------------------------------------------------2.1
\rm{Let $R = \oplus_{n \geq 0}R_n$  be a homogeneous
Noetherian ring, so that $R$ is positively graded, $R_0$ is
Noetherian and $R = R_0[l_0,\cdots,l_r]$ with finitely many elements
$l_0,\cdots,l_r \in R_1$. Let $R_+$ denote the irrelevant ideal
$\oplus_{n > 0}R_n$ of $R$.}  \hfill $\bullet $
\end{nota}

\begin{remi}\label{2.2}%-------------------------------------------------------------------------------------------------2.2

\textit{(Local cohomology and Castelnuovo-Mumford regularity)}
\rm{ \noindent(A)  Let $i \in {\N} _0:= \{0, 1, 2, \cdots \}.$ By
$H_{R_+}^i(\bullet)$ we denote the $i$-th local cohomology functor
with respect to $R_+$. Moreover by $D_{R_+}^i(\bullet)$ we denote
the $i$-th right derived functor of the ideal transform functor
$\displaystyle D_{R_+}(\bullet)= \lim _ {_{n \rightarrow \infty
}}\Hom_R ((R_+)^n,\bullet)$ with respect to $R_+$.

\smallskip

\noindent(B) Let  $M:= \oplus_{n \in {\Z}}M_n$ be a  graded
$R$-module. Keep in mind that in this situation the $R$-modules
$H_{R_+}^i(M)$ and $D_{R_+}^i(M)$ carry natural gradings. Moreover
we then have a natural exact sequence of graded $R$-modules

\smallskip

\begin{enumerate}
\item[\ (i)]\hskip1cm $0\longrightarrow H_{R_+}^0(M)\longrightarrow
M\longrightarrow D_{R_+}^0(M)\longrightarrow
H_{R_+}^1(M)\longrightarrow 0$
\end{enumerate}

\smallskip

and natural isomorphisms of graded $R$-modules

\smallskip

\begin{enumerate}
\item[(ii)]\hskip1cm$D_{R_+}^i(M) \cong H_{R_+}^{i+1}(M) \ \ \
\text{for all}\ \ i
> 0.$
\end{enumerate}

\smallskip

\noindent(C) If $T$ is a graded $R$-module and $n\in {\Z }$, we use
$T_n$ to denote the $n$-th graded component of $T$. In particular,
we define the} {\it beginning} and the \it{end} \rm{of $T$
respectively by

\smallskip

\begin{enumerate}
\item[(i)]\hskip1cm$\beg(T):= \inf\{ n\in {\Z }| T_n \ne 0\}$,
\end{enumerate}

\smallskip

\begin{enumerate}
\item[(ii)]\hskip1cm$\en(T):= \sup\{ n\in {\Z }| T_n \ne 0\}$.
\end{enumerate}

\smallskip

with the standard convention that $\inf \emptyset = \infty$ and}
$\sup \emptyset = -\infty$.

\smallskip

\noindent(D) If the graded $R$-module $M$ is finitely generated, the
$R_0$-modules $H_{R_+}^{i}(M)_n$ are all finitely generated and
vanish as well for all $n \gg 0$ as for all $i > \dim(M)$. So, we have
\[-\infty \leq a_i(M):= \en(H_{R_+}^{i}(M)) < \infty \ \ \text{for all} \ \ i \geq 0\]
with $a_i(M):= -\infty$ for all $i> \dim(M).$

\smallskip

If $k \in {\N} _0$, the \it{Castelnuovo-Mumford regularity of} $M$
\it{at and above level} $k$ \rm{is defined by}

\smallskip

\begin{enumerate}
\item[ ] \hskip1cm$\reg^k(M): = \sup \{a_i(M) + i|\  i\geq k\} \,\,\, (< \infty ).$
\end{enumerate}

\smallskip

The \it {Castelnuovo-Mumford regularity} of $M$ \rm{is
defined by} $\reg(M):= \reg^0(M).$

\noindent(E) We also shall use the {\it generating degree} of $M$, which is defined by
\[\gendeg(M)= \inf \{n\in \Z\mid \ M= \displaystyle \Sigma_{m\leq n}R M_m\}.\]
If the graded $R$-module $M$ is finitely generated, we have $\gendeg(M)\leq \reg(M)$. \hfill $\bullet $

\end{remi}

\begin{remi}\label{2.3} \it{(Cohomological Hilbert functions)}%--------------------------------------------------------------2.3
\rm{  (A) Let $i\in {\N}_0$ and assume that the base ring $R_0$ is
Artinian. Let $M$ be a finitely generated graded $R$-module. Then,
the graded $R$-modules $H_{R_+}^{i}(M)$ are Artinian. In particular for all $i\in {\N}_0$ and all $n\in {\Z}$
we may define the non-negative integers

\smallskip

\begin{enumerate}
\item[(i)]\hskip1cm$h_M^i(n):= \length_{R_0}(H_{R_+}^{i}(M)_n)$,
\end{enumerate}

\smallskip

\begin{enumerate}
\item[(ii)]\hskip1cm$d_M^i(n):= \length_{R_0}(D_{R_+}^{i}(M)_n).$
\end{enumerate}

\smallskip

Fix $i\in {\N}_0$. Then the functions

\begin{enumerate}
\item[(iii)]\hskip1cm$h_M^i:{\Z}\rightarrow {\N}_0, \ \ n\mapsto h_M^i(n)$,
\end{enumerate}

\smallskip

\begin{enumerate}
\item[(iv)]\hskip1cm$d_M^i:{\Z}\rightarrow {\N}_0, \ \ n\mapsto d_M^i(n)$
\end{enumerate}

are called the $i$-th \it {Cohomological Hilbert functions} \rm{of
the} \it{first} \rm{respectively the} \it{second kind} \rm{of} $M$.

\smallskip

\noindent(B) Let $M$ be a
finitely generated graded $R$-module and let $x\in R_{1}$. We also
write $\mathit{\Gamma}_{R_{+}}(M)$ for the $R_{+}$-torsion submodule
of $M$ which we identify with $H^{0}_{R_{+}}(M)$. By NZD$_{R}(M)$
resp. ZD$_{R}(M)$ we denote the set of non-zerodivisors resp. of
zero divisors of $R$ with respect to $M$.
The linear form $x\in R_{1}$ is said to be ($R_{+}$-) \it{filter
regular with respect to }$M$  \rm{if} $x\in $NZD}$_{R}(M/\it{\Gamma}_{R_{+}}(M))$.     \hfill $\bullet$

\end{remi}

\begin{remi}\label{2.4}%----------------------------------------------------------------------------------------------------------2.4
{\rm (cf. [6, Definition 5.2]) For $d\in {\N}$ let $\D^d$ denote
the class of all pairs $(R, M)$ in which $R = \oplus_{n\in
{\N}_0}R_n$ is a Noetherian homogenous ring with Artinian base
ring $R_0$ and $M = \bigoplus_{n\in \mathbb{Z}}M_n$ is a finitely
generated graded $R$-module with $\dim(M)\leq d$.\hfill $\bullet
$}

\end{remi}

\vskip 1 cm
%%%%%%%%%%%%%%%%%%%%%%%%
\section{Finiteness and Boundedness of Cohomology}%-----------------------------------------------------------------------------------------sec.3
%%%%%%%%%%%%%%%%%%%%%%%%

{\rm We keep the notations and hypotheses introduced in Section
2.}

\begin{defin}\label{6.1 Definition}%--------------------------------------------------------------------------------------------------------------------------------------3.1

\rm{The {\it cohomology table} of the pair $(R, M)\in \D^d$ is the
family of non-negative integers}
\[ d_{M}:= (d^i_{M}(n))_{(i,n) \in{\mathbb N}_0 \times {\mathbb Z}}.\]\hfill $\bullet $

\end{defin}

\begin{remi}\label{6.2 Reminder}%--------------------------------------------------------------------------------------------------------------------------------------6.2

\rm{(A) According to [5] the {\it cohomological pattern}
${\mathcal P}_{M}$ of the pair $(R, M) \in \D^d$ is defined as
the set of places at which the cohomology table of $(R, M)$ has a
non-zero entry:
\[{\mathcal P}_{M}:= \{ (i,n) \in {\N}_0 \times{\mathbb Z} \big\arrowvert d^i_{M}(n) \not= 0 \}.\]

\smallskip

(B) A set $P \subseteq {\mathbb N}_0 \times {\mathbb Z}$ is called a
{\it tame combinatorial pattern of width} $w \in {\mathbb N}_0$ if
the following conditions are satisfied:

\smallskip

%\renewcommand{\descriptionlabel}[1]%
 %            {\hspace{\labelsep}\textrm{#1}}
%\begin{description} \setlength{\labelwidth}{13mm}
%\setlength{\labelsep}{1.3mm} \setlength{\itemindent}{0mm}

\smallskip

\begin{enumerate}
\item[$(\pi _1)$]\hskip1cm$\exists m, n \in {\mathbb Z} : (0,m), (w,n) \in
P$;
\end{enumerate}

\smallskip

\begin{enumerate}
\item[$(\pi _2)$]\hskip1cm$(i,n) \in P \Rightarrow i \leq w$;
\end{enumerate}

\smallskip

\begin{enumerate}
\item[$(\pi _3)$]\hskip1cm$(i,n) \in P \Rightarrow \exists j \leq i : (j, n
+ i - j + 1) \in P$;
\end{enumerate}

\smallskip

\begin{enumerate}
\item[$(\pi _4)$]\hskip1cm$(i,n) \in P \Rightarrow \exists k \geq i : (k,
n + i - k - 1) \in P$;
\end{enumerate}

\smallskip

\begin{enumerate}
\item[$(\pi _5)$]\hskip1cm$i > 0 \Rightarrow \forall n \gg 0 : (i,n) \notin
P$;
\end{enumerate}

\smallskip

\begin{enumerate}
\item[$(\pi _6)$]\hskip1cm$\forall i \in {\mathbb N} : (\forall n \ll 0 ; (i,n)
\in P) \mbox{ \it or else } (\forall n \ll 0 : (i,n) \notin P)$.
\end{enumerate}

\smallskip

By [5] we know:}

\smallskip

\begin{enumerate}
\item[(a)] If $(R, M) \in\D^d$ with $\dim
(M) = s>0$, then ${\mathcal P}_{M}$ is a tame combinatorial pattern
of width $w= s-1.$
\end{enumerate}

\smallskip

\begin{enumerate}
\item[(b)] If $P$ is a tame combinatorial pattern of width $w \leq
d-1$, then there is a pair $(R, M) \in\D^d$ such that the base ring
$R_0$ is a field and $P = {\mathcal P}_{M}$.\hfill $\bullet $

\end{enumerate}

\smallskip

\end{remi}

\smallskip

{\rm By the previous observation, the set of patterns $\{
{\mathcal P}_{M} \big\arrowvert (R, M) \in \D^d \}$ is quite
large, and hence so is the set of cohomology tables $\{ d_{M}
\big\arrowvert (R, M) \in \D^d \} $. Therefore, one seeks for
decompositions $\bigcup _{i \in {\mathbb I}} {\mathcal C}_i =
\D^d$ of $\D^d$ into ``simpler'' subclasses ${\mathcal C}_i$ such
that for each $i \in {\mathbb I}$ the set $\{ d_{M}
\big\arrowvert (R, M) \in {\mathcal C}_i \} $ is finite. Bearing
in mind this goal, we define the following concepts:}

\medskip

\begin{defins}\label{6.3 Definitions}%--------------------------------------------------------------------------------------------------------------------------------------6.3
\rm{(A) Let ${\mathcal C} \subseteq \D^d$ be a subclass. We say that
${\mathcal C}$ {\it is a subclass of finite cohomology} if
\[ \sharp \{ d_{M} \big\arrowvert (R, M) \in {\mathcal C} \} < \infty.\]

\smallskip

(B) We say that ${\mathcal C} \subseteq \D^d$ is a} {\it subclass of
bounded cohomology} if
\[ \forall (i, n) \in {\mathbb N}_0 \times {\mathbb Z} : \sharp\{ d^i _{M} (n) \big\arrowvert (R, M) \in
{\mathcal C} \} < \infty.\]  \hfill $\bullet $

\end{defins}

\begin{rem}\label{6.4Remark}%--------------------------------------------------------------------------------------------------------6.4

\rm{(A) Let ${\mathcal C}, {\mathcal D} \subseteq \D^d$ be
subclasses of $\D^d$. Then clearly}

\smallskip

\renewcommand{\descriptionlabel}[1]%
             {\hspace{\labelsep}\textrm{#1}}
\begin{description}
\setlength{\labelwidth}{13mm} \setlength{\labelsep}{1.3mm}
\setlength{\itemindent}{0mm}

\item[(a)] \rm{If ${\mathcal C} \subseteq {\mathcal D}$ and ${\mathcal D}$
is of finite cohomology or of bounded cohomology, then so is ${\mathcal C}$ respectively.}
\end{description}

\smallskip

\noindent (B) \rm{If $r\in \Z$, we have a bijection}

\smallskip

\renewcommand{\descriptionlabel}[1]%
             {\hspace{\labelsep}\textrm{#1}}
\begin{description}
\setlength{\labelwidth}{13mm} \setlength{\labelsep}{1.3mm}
\setlength{\itemindent}{0mm}

\item[ ] $ \{ d_{M} \big\arrowvert (R, M) \in
{\mathcal C} \} \ \twoheadrightarrow\{ d_{M(r)} \big\arrowvert
(R, M) \in {\mathcal C}\} $ \rm{given by}  $d_{M} \mapsto
d_{M(r)}$.
\end{description}\hfill $\bullet $
 \end{rem}
\smallskip

{\rm Now, we show how the finiteness and boundedness conditions defined
above are related.}

\begin{thm}\label{6.5 Theorem}%--------------------------------------------------------------------------------------------------------------------------------------6.5

For a subclass ${\mathcal C} \subseteq \D^d$ the following
statements are equivalent:

\smallskip

\renewcommand{\descriptionlabel}[1]%
             {\hspace{\labelsep}\textrm{#1}}
\begin{description}
\setlength{\labelwidth}{13mm} \setlength{\labelsep}{1.3mm}
\setlength{\itemindent}{0mm}

\item[{\rm (i)}] ${\mathcal C}$ is a class of finite cohomology.

\item[{\rm (ii)}] ${\mathcal C}$ is a class of bounded cohomology.

\item[{\rm (iii)}] For each $n_0 \in {\mathbb Z}$ the set $\triangle_{\C, n_0}:= \{d^i_{M}(n_0-i)\mid \ (R, M)\in \C, \ 0\leq i< d\}$ is finite.

\item[{\rm (iv)}] There is some $n_0 \in {\mathbb Z}$ such that the set $\triangle_{\C, n_0}$ of statement (iii) is finite.
\end{description}
\end{thm}

\begin{proof}
The implications $(\rm{i}) \Rightarrow (\rm{ii}) \Rightarrow
(\rm{iii}) \Rightarrow (\rm{iv})$ are clear from the definitions.
To prove the implication $(\rm{iv}) \Rightarrow (\rm{i})$ fix $n_0
\in {\mathbb Z}$ and assume that the set $\triangle_{\C, n_0}$ is
finite. Then there is some non-negative integer $h$ such that
$d^i_{M(n_0)}(-i) \leq h$ for all pairs $(R, M) \in {\mathcal C} $
and all $i \in \{ 0, \cdots , d-1 \} $. By [6, Theorem 5.4] it
thus follows that the set of functions $$\{d^i_{M(n_0)}\mid \ (R, M)\in \C, i\in \N_0\}$$
 is  finite. By Remark 3.4 (B) we now may conclude
that the class ${\mathcal C}$ is of finite cohomology.
 \end{proof}

\sm

{\rm So, by Theorem 3.5 boundedness and finiteness of cohomology are the same for a given class $\C\subseteq {\D}^d$.}

\begin{defin}\label{6.6 Definition}%--------------------------------------------------------------------------------------------------------------------------------------6.6

\rm{Let $d \in {\mathbb N}_0$, let ${\mathcal C} \subseteq
{\mathcal D}^d$ and let ${\mathbb S} \subseteq \{ 0, \cdots,
d-1\} \times {\mathbb Z}$ be a subset. We say that the set}
${\mathbb S}$ {\it bounds cohomology in ${\mathcal C}$} if for
each family $(h^\sigma ) _{\sigma \in {\mathbb S}}$ of non
negative integers $h^\sigma $ the class
\[ \{ (R, M) \in {\mathcal C}\mid \ \forall
(i,n) \in {\mathbb S} : d^i_{M}(n) \leq h^{(i,n)}\}\] is of finite cohomology. \hfill $\bullet $
\end{defin}

\begin{rem}\label{6.7 Remark}%--------------------------------------------------------------------------------------------------------------------------------------6.7

\rm{(A) Let $d \in {\mathbb N}_0$, let ${\mathcal C}, {\mathcal D}
\subseteq \D^d$ and ${\mathbb S}, {\mathbb T} \subseteq \{ 0,
\cdots , d-1 \} \times {\mathbb Z}$. Then obviously we can say}

\smallskip

\renewcommand{\descriptionlabel}[1]%
             {\hspace{\labelsep}\textrm{#1}}
\begin{description}
\setlength{\labelwidth}{13mm} \setlength{\labelsep}{1.3mm}
\setlength{\itemindent}{0mm}

\item[ ] {\rm If ${\mathbb S} \subseteq {\mathbb T}$ and ${\mathbb S}$
bounds cohomology in ${\mathcal C}$, then so does ${\mathbb T}$.}
\end{description}

\smallskip

(B) \rm{If $r \in {\mathbb Z}$, we can form the
set }
$ {\mathbb S}(r):= \{ (i,n + r) \big\arrowvert (i,n) \in
      {\mathbb S} \}.$
\rm{In view of the bijection of Remark 3.4 (B) we have}

\smallskip

\renewcommand{\descriptionlabel}[1]%
             {\hspace{\labelsep}\textrm{#1}}
\begin{description}
\setlength{\labelwidth}{13mm} \setlength{\labelsep}{1.3mm}
\setlength{\itemindent}{0mm}

\item[ ] {\rm ${\mathbb S}(r)$ bounds cohomology in ${\mathcal C}(r):= \{(R, M(r))\mid \ (R, M)\in \C\}$
if and only if ${\mathbb S}$ does in $\C$.}
\end{description}

\smallskip

\rm{(C)
For all $s \in \{ 0, \cdots , d \} $ we set
\[ {\mathbb S}^{< s} := {\mathbb S} \cap (\{ 0, \cdots , s-1)\times {\mathbb Z}).\]
as $\D^s \subseteq \D^d$ it follows easily:}

\smallskip

\renewcommand{\descriptionlabel}[1]%
             {\hspace{\labelsep}\textrm{#1}}
\begin{description}
\setlength{\labelwidth}{13mm} \setlength{\labelsep}{1.3mm}
\setlength{\itemindent}{0mm}

\item[ ] {\rm If ${\mathbb S}$ bounds cohomology in $\C$,
then ${\mathbb S}^{< s}$ bounds cohomology in} $\D^s \cap
\C$.
\end{description}
  \hfill $ \bullet $
\end{rem}

\begin{cor}\label{6.8 Corollary}%--------------------------------------------------------------------------------------------------------------------------------------6.8

Let ${\mathcal C} \subseteq \D^d$ and $n \in {\mathbb Z}$. Then, the
"$n$-th diagonal"
\[ \{ (i, n - i )\big\arrowvert i = 0, \cdots , d-1 \}\]
bounds cohomology in ${\mathcal C}$.

\end{cor}

\begin{proof} This is immediate by Theorem 3.5.
\end{proof}

\vskip 1 cm

%%%%%%%%%%%%%%%%%%%%%%%%
\section{Quasi-Diagonals}%----------------------------------------------------------------------------------------------------------------sec.7
%%%%%%%%%%%%%%%%%%%%%%%%

{\rm Our first aim is to generalize Corollary 3.8 by showing that not
only the diagonals bound cohomology on ${\mathcal C}$, but rather
all ``quasi-diagonals''. We shall define below, what such a
quasi-diagonal is.}

\smallskip

\begin{lem}\label{7.1 Lemma}%-------------------------------------------------------------------------------------------------------------------------7.1

Let $t \in \{ 1, \cdots , d \} $, let $(n_i)^{d-1}_{i = d - t}$ be a
sequence of integers such that $n_{d - 1} < \ldots < n_{d - t}$ and
let ${\mathcal C} \subseteq \D^d$ be a class such that the set $\{ d^i_M(n_i) \big\arrowvert \ (R, M)\in \C \}$ is finite for all
$i\in \{d-t, \cdots d-1\}$.
 Then the set $\{ d^i_M(n) \big\arrowvert \ (R, M)\in \C \}$ is finite whenever
$n_i \leq n$ and $d - t \leq i \leq d-1$.
\end{lem}

\begin{proof}
By our hypothesis there is some $h \in {\mathbb N}_0$ with
$d^i_{M}(n_i) \leq h$ for all $i \in \{ d - t, \cdots , d-1 \} $ and
all pairs $(R, M) \in {\mathcal C}$.

On use of standard reduction arguments
we can restrict ourselves to the case where the Artinian base ring $R_0$
is local with infinite residue field. Let $(R, M)\in {\mathcal
C}$. Replacing $M$ by $M/\Gamma_{R_+}(M)$ we may assume that $M$ is
$R_+$-torsion free. Therefore, there exists $x\in R_1\cap
\mbox{NZD}(M)$. For each $i\in {\N}_0$ and $m\in \Z$, the short
exact sequence $0\longrightarrow M(-1)\longrightarrow
M\longrightarrow M/xM\longrightarrow 0$ induces long exact sequences

\[ D^i_{R_+}(M)_{m - 1} \rightarrow D^i_{R_+}(M)_{m}
\rightarrow D^i_{R_+}(M/xM)_{m}\longrightarrow D^{i+1}_{R_+}(M)_{m
-1}. \leqno{(\ast _{i, m})}\]

As $\dim(M/xM)< d$, the sequences
$(*_{d-1, m})$ imply that $d^{d-1}_M(m)\leq d^{d-1}_M(m-1)$ for
all $m\in \Z$. This proves our claim if $t=1$. So, let $t>1$.

Assume inductively that the set $\{ d^i_M(n_i) \big\arrowvert \
(R, M)\in \C \}$ is finite whenever $n_i \leq n$ and  $d - t + 1
\leq i \leq d -1$. It remains to find a family of non-negative
integers $(h_n)_{n \geq n_{d - t}}$ such that $d^{d - t} _{M}(n)
\leq h_n$ for all $n \geq n_{d - t}$. Let $\mathcal{E}$ denote the class
of all pairs $(R, M/xM)=(R, \overline{M})$ in which $(R, M)\in{\mathcal C}$ and
$x\in R_1\cap \mbox{NZD}(M)$. As $n_i - 1 \geq n_{i + 1}$ for all
$i \in \{ d - t, \cdots , d-2\} $, the sequences $(\ast
_{i,n_i})$ show that $$d^i_{M/xM} (n_i) \leq d^{i + 1}_M(n_i - 1)
+ h  \ \mbox{for} \ i \in \{ d - t, \cdots , d - 2 \}.$$ This
means that the set $\{ d^i_{\overline{M}}(n_i) \big\arrowvert \ (R,
\overline{M})\in \mathcal{E} \}$ is finite whenever
 $(d-1) - (t - 1) \leq i \leq d -2$.
So, by induction  the set $\{ d^i_{\overline{M}}(n_i) \big\arrowvert \ (R, \overline{M})\in \mathcal{E} \}$
is finite whenever
$n_i \leq n$ and  $(d-1) - (t - 1) \leq i \leq d -2$.

In particular there is a family of non-negative integers $(k_m)_{m
\geq n _{d - t}}$ such that $d^{d - t}_{M/xM}(m) \leq k_m$ for
all $m \geq n_{d - t}$. Now, for each $n \geq n_{d - t}$ set
$h_n:= h + \Sigma _{n_{d - t} < m \leq n} k_m$. If we choose $(R,
M) \in {\mathcal C}$, the sequences $(\ast _{d - t, n})$ imply
that $d^{d-t}_{M}(n) \leq h_n$ for all $n \geq n_{d - t}$.
\end{proof}

\begin{prop}\label{7.2 Proposition}%-------------------------------------------------------------------------------------------------------------------------7.2

Let $(n_i)^{d-1}_{i = 0}$ be a sequence of integers such that $
n_{d - 1} < \ldots < n_0$ and let $\C\subseteq {\D}^d$. Then the set $\{ (i, n_i)
\big\arrowvert i = 0, \cdots ,d-1\} $ bounds cohomology in $\C$.
\end{prop}

\begin{proof} Let $(h^i)_{i=0}^{d-1}$ be a family of non-negative integers  and let $\C'$
 be the class of all pairs $(R, M)\in \C$ such that $d^i_M(n_i)\leq h^i$
 for $ i = 0, \cdots , d-1$. Then, by Lemma
4.1
 the set $\{d^i_M(n)\mid \ (R, M)\in \C'\}$ is finite, whenever $n\geq n_i$ and $0\leq i\leq d-1$.
Therefore the  set $\triangle_{\C', n_0}:=\{d^i_M(n_0-i)\mid \
(R, M)\in \C', \ 0\leq i< d\}$  is finite. So, by Theorem 3.5 the
class $\C'$ is of finite cohomology. It follows that $\{ (i,n_i)
\big\arrowvert i = 0, \cdots , d-1\} $ bounds cohomology in
${\mathcal C}$.
\end{proof}

\begin{defin}\label{7.3 Definition}%-------------------------------------------------------------------------------------------------------------------------7.3

\rm{A set ${\mathbb T} \subseteq \{ 0, 1, \cdots , d-1 \} \times
{\mathbb Z}$ is called a {\it quasi-diagonal} if there is a sequence
of integers $(n_i)^{d-1}_{i=0}$ such that $n_{d-1} < n_{d-2}<\ldots
< n_0$ and}
\[ {\mathbb T} = \{ (i, n_i) \big\arrowvert i = 0, \cdots , d-1\}.\]   \hfill $\bullet$

\end{defin}

\smallskip

{\rm Observe, that diagonals in $\{ 0, \cdots , d-1\} \times {\mathbb Z}$
are quasi-diagonals. So, the next result generalizes Corollary  3.8.}

\smallskip

\begin{cor} \label{7.4 Corollary}%-------------------------------------------------------------------------------------------------------------------------7.4

Let ${\mathbb S} \subseteq \{ 0,1,\cdots , d\} \times {\mathbb Z}$
be a set which contains a quasi-diagonal. Then ${\mathbb S}$
bounds cohomology in each subclass $\C\subseteq \D^d$.
\end{cor}

\begin{proof}
Clear by Proposition 4.2.
\end{proof}

\smallskip

{\rm Our next goal is to show that the converse of Corollary 4.4
 holds, namely: if a set ${\mathbb S} \subseteq \{ 0,1, \cdots , d-1
\} \times {\mathbb Z}$ bounds cohomology in $\D^d$, then ${\mathbb
S}$ contains a quasi-diagonal.}

\smallskip

\begin{remi}\label{4.5 Reminder}%--------------------------------------------------------------------------------------------
{\rm Let $K$ be a field, let $R= K\oplus R_1\oplus \cdots$ and
$R^{'}= K\oplus R^{'}_1\oplus \cdots$ be two Noetherian
homogeneous $K$-algebras. Let $R\boxtimes _K R^{'}:= K\oplus
(R_1\otimes R^{'}_1)\oplus (R_2\otimes R^{'}_2)\oplus \cdots
\subseteq R\otimes_K R^{'}$
 be the {\it Segre product ring of $R$ and $R^{'}$}, a Noetherian homogeneous $K$-algebra. For a graded $R$-module $M=\bigoplus_{n\in \Z}M_n$
 and a graded $R^{'}$-module $M^{'}=\bigoplus_{n\in \Z}M^{'}_n$ let
 $M\boxtimes _K M^{'}:=\bigoplus _{n\in \Z}M_n\otimes _K M^{'}_n\subseteq M\otimes _K M^{'}$
 the Segre product module of $M$ and $M^{'}$, a graded $R\boxtimes _K R^{'}$-module.
  Keep in mind, that the {\it K\"unneth relations} (for Segre products)
yield isomorphism of graded $R\boxtimes _K R^{'}$-modules
\[D^i_{(R\boxtimes _K R^{'})_+}(M\boxtimes _K M^{'})\cong \bigoplus _{j=0}^{i}D^j_{R_+}(M)\boxtimes _K D^{i-j}_{R^{'}_+}(M^{'})\]
for all $i\in {\N}_0$ (cf. [23], [14], [20]).}  \hfill $\bullet$
\end{remi}

\smallskip

\begin{lem}\label{7.6 Lemma}%-------------------------------------------------------------------------------------------------------------------------7.5

Let $d>1$ and set $R:= K[x_1, \cdots, x_d]$ be a polynomial ring
over some infinite field $K$. Let ${\mathbb S} \subseteq \{ 0, 1,
\cdots , d-1 \} \times {\mathbb Z}$ such that
\renewcommand{\descriptionlabel}[1]%
            {\hspace{\labelsep}\textrm{#1}}
\begin{description}
\setlength{\labelwidth}{13mm} \setlength{\labelsep}{1.3mm}
\setlength{\itemindent}{0mm}
\item[\rm{(1)}]
${\mathbb S}$ contains no quasi-diagonal,
\item[\rm{(2)}]
$ {\mathbb S} \cap (\{ 0,\cdots , d - 2
\}\times \Z)$
%\times {\mathbb Z})$
contains a quasi-diagonal $\{(i, n_i)\mid \ i=0, \cdots, d-2\}$
%\end{description}
and
%\smallskip
%\renewcommand{\descriptionlabel}[1]%
 %            {\hspace{\labelsep}\textrm{#1}}
%\begin{description}
%\setlength{\labelwidth}{13mm} \setlength{\labelsep}{1.3mm}
%\setlength{\itemindent}{0mm}
\item[\rm{(3)} ]
$ {\mathbb S} \cap (\{ d-1\} \times {\mathbb Z})$
$\not= \emptyset $.
\end{description}

\smallskip

Then

\renewcommand{\descriptionlabel}[1]%
             {\hspace{\labelsep}\textrm{#1}}
\begin{description}
\setlength{\labelwidth}{13mm} \setlength{\labelsep}{1.3mm}
\setlength{\itemindent}{0mm}

\item[{\rm (a)}] $(d-1,n) \notin {\mathbb S}$ for all $n \ll 0$,

\item[{\rm (b)}] There is a family $(M_k)_{k \in {\mathbb N}}$ of finitely generated graded  $R$-modules, locally free of
rank $\leq ((d-1)!)^2$ on $\Proj(R)$ such that the set
$\{d^i_{M_k}(n)\mid \ k\in \N\}$ is finite for all $(i, n)\in
{\mathbb S}$ and
\[ \lim _{k \rightarrow \infty } d^{d-1}_{M_k}(r)= \infty , \mbox{ where } \ r := \inf
\{ n \in {\mathbb Z} \big\arrowvert (d-1,n) \in {\mathbb S}\}- 1.\]
\end{description}
\end{lem}

\begin{proof} For all $i \in \{1, \cdots, d\}$ we write $R^i:=
K[x_1, \cdots, x_i]$ and ${\mathbb S}^i := {\mathbb S} \cap (\{ i\}
\times {\mathbb Z})$. Statement (a) follows immediately from our
hypotheses on the set ${\mathbb S}$. So, it remains to prove
statement (b). After shifting appropriately we may assume that $r =
- 1$.

\smallskip

By our hypotheses on ${\mathbb S}$ it is clear that ${\mathbb S}^i
\not= \emptyset $ for all $i \in \{ 0, \cdots , d-1\} $. Let
\[\alpha_i:= \sup \{n\in \Z\mid \ (i, n)\in {\mathbb S}^i\} \ \mbox{for all} \ i\in \{0, \cdots, d-1\}.\]
 Then by our hypothesis on $\mathbb S$ we have  $\alpha_i < \infty $ for some $i \in \{ 1, \cdots , d - 2
\} $. Let
\[ s:= \min \{ i \in \{ 0, \cdots , d - 2 \} \big\arrowvert \alpha_i < \infty \}\]
and
\[ n_s:= \max \{ n \in {\mathbb Z} \big\arrowvert (s,n) \in
{\mathbb S}^s \}.\]

Now, we may find a quasi-diagonal $\{ (i,n_i) \big\arrowvert i = 0,
\cdots , d - 2 \} $ in ${\mathbb S}\cap (\{0, \cdots, d-2\}\times \Z)$ such that for all
$i \in \{ s + 1, \cdots , d - 2 \} $ we have
\[ n_i = \max \{ n < n_{i - 1} \big\arrowvert (i,n) \in{\mathbb S} \}.\]

As ${\mathbb S}$ contains no quasi-diagonal, we must have $n_{d - 2}
\leq 0$. For all $m, n \in {\mathbb Z} \cup \{ \pm \infty \} $ we
write $] m,n [ := \{ t \in {\mathbb Z} \big\arrowvert  m < t < n \}
$. Using this notation we set
\[ t_{-1}:=  \infty ; \ t_{d-s-1}:= -\infty ; \ t_i := \max \{ d - s -
i - 2, n_{i + s}\} , \ \forall i \in \{ 0, \cdots , d - s-2 \}\]
and write
\[ P := \bigcup ^{d - s-1}_{i = 0} (\{ i \} \times ] t_i, t_{i - 1}[ ).\]

Observe, that by our choice of the pairs $(i,n_i)$ we
have
\[ \mbox{if } s \leq i \leq d-1 \mbox{ and } (i,n) \in {\mathbb S},
\mbox{ then } (i - s, n) \notin P. \leqno{(\ast )} \]
 Moreover by [5, 2.7] the set $P \subseteq \{ 0, \cdots , d -
s-1\} \times {\mathbb Z}$ is a minimal combinatorial pattern of
width $d - s-1$. So, by [5, Proposition 4.5], there exists a
finitely generated  $R^{d-s}$-module $N$, locally free of rank
$\le (d - s-1)$! on $\Proj(R^{d-s})$ such that ${\mathcal P}_{N}
= P$.

\smallskip

Now, consider the Segre product ring $S:= R^{s+1}\boxtimes _ K
R^{d-s}$ and for each $k \in {\mathbb N}$ let $M_k$ be the
finitely generated graded $S$-module $R^{s+1}(-k)\boxtimes _K N$,
which is locally free  of rank $\leq (d-1)!/s!$ on $\Proj(S)$.
Observe that
\[ d^j_{R^{s+1}} \equiv 0
\mbox{ for all } j \not= 0,s \mbox{ and } d^l _{N} \equiv 0
\mbox{ for all } l > d - s-1.\]

Now, we get from the K\"unneth relations (cf. Reminder 4.5) for all $i \in \{ 0, \cdots
, d-1\}$ and all $n \in {\mathbb Z}$
\[ d^i_{M_k}(n) = \begin{cases}
d^0_{R^{s+1}}(-k + n)d^i_{N}(n) &\mbox{for } 0 \leq i < s  \\
d^0_{R^{s+1}}(-k + n)d^i_{N}(n) + d^s_{R^{s+1}}(- k + n)d^{i-s}_{N}(n)&\mbox{for } s \leq i \leq d - s-1 , \\
d^s_{R^{s+1}}(- k + n)d^{i - s}_{N}(n) &\mbox{for } d - s -1< i \leq
d-1.\end{cases}\]

As $P = {\mathcal P}_{N}$ and in view of $(\ast )$ we have $d^{i -
s} _{N} (n) = 0$ for all $(i,n) \in {\mathbb S}$ with $s \leq i \leq
d-1$. Moreover, for all $n \in {\mathbb Z}$ and all $k \in {\mathbb
N}$ we have $d^0_{R ^{s+1}}(- k + n) \leq d^0_{R^{s+1}} (n - 1)$. So
for all $k \in {\mathbb N}$ and all $(i, n) \in {\mathbb S}$ we get
\[ d^i_{M_k}(n) \begin{cases}
\leq d^0_{R^{s+1}}(n - 1)  d^i_{N}(n), &\mbox{for } 0 \leq i \leq d - s-1 , \\
= 0 , &\mbox{if } d - s -1< i \leq d-1.
\end{cases}\]

Therefore the set $\{d^i_{M_k}(n)\mid \ k\in \N\}$ is finite for all $(i, n)\in \mathcal{S}$.

\smallskip

Moreover $d^{d-1}_{M_k}(-1) = d^s_{R^{s+1}}(-k - 1)  d^{d -
s-1}_{N}(-1)$. As $(d - s-1, - 1) \in P$ we have $d^{d - s-1}
_{N}(-1)
> 0$ and hence $d^s_{R^{s+1}}(-k-1)= \begin{pmatrix}
  {k} \\
  {s}
\end{pmatrix}$ implies that
\[ \lim _{k \rightarrow \infty } d^{d-1}_{M_k}(-1) = \infty.\]

As $\dim(S)= d$, there is a finite injective morphism
$R\longrightarrow S$ of graded rings, which turns $S$ in an
$R$-module of rank $(d-1)!/s!(d-s-1)!$. So $M_k$ becomes an
$R$-module locally free of rank $\leq
[(d-1)!/s!(d-s-1)!][(d-1)!/s!]\leq ((d-1)!)^2$ on $\Proj(R)$.
Moreover, by Graded Base Ring Independence of Local Cohomology,
we get isomorphisms of graded $R$-modules $D^j_{S_+}(M_k)\cong
D^j_{R_+}(M_k)$ for all $j\in {\N}_0$. Now, our claim follows
easily.
\end{proof}

\begin{defin}\label{7.7 Definition}%-------------------------------------------------------------------------------------------------------------------------7.6

{\rm A class ${\mathcal D} \subseteq {\mathcal D}^d$ is said to
be {\it big}, if for each $t \in \{ 1, \cdots , d\} $ there is an
infinite field $K$ such that $\D$ contains all pairs $(R, M)$ in which
$R$ is the polynomial ring $K[x_1, \cdots, x_t]$. }  \hfill $\bullet$
\end{defin}

\begin{prop}\label{7.8 Proposition}%-------------------------------------------------------------------------------------------------------------------------7.7

Let ${\mathcal C} \subseteq \D^d$ be a big class and let ${\mathbb
S} \subseteq \{ 0, \cdots , d-1 \} \times {\mathbb Z}$ be a set
which bounds cohomology in ${\mathcal C}$. Then ${\mathbb S}$
contains a quasi-diagonal.
\end{prop}

\begin{proof} There is an infinite field $K$ such that with $R:= K[x_1, \cdots,
x_d]$ we have $(R, R(-k)) \in {\mathcal C}$ for all $k \in
{\mathbb N}$. The  set $\{d^i_{R(-k)}(n)\mid \ k\in \N\}$ is
finite for all $(i, n)\in \{ 0, \cdots , d - 2 \} \times {\mathbb
Z}$ and $\displaystyle \lim _{k \rightarrow \infty }
d^{d-1}_{R(-k)}(0) = \infty $. It follows that ${\mathbb
S}^{d-1}:= {\mathbb S} \cap (\{ d -1\} \times {\mathbb Z}) \not=
\emptyset $. This proves our claim if $d = 1$.

\smallskip

 So, let $d > 1$. Clearly $\D
^{d - 1} \cap {\mathcal C} \subseteq \D^{d - 1}$ is a big class
and ${\mathbb S}^{< (d - 1)} = {\mathbb S} \cap (\{ 0, \cdots , d
- 2 \} \times {\mathbb Z})$ bounds cohomology in $\D^{d-1} \cap
{\mathcal C}$ (s. Remark 3.7 (C)). So, by induction the set
${\mathbb S}^{<(d - 1)}$ contains a quasi-diagonal. If ${\mathbb
S}$ would contain no quasi-diagonal, Lemma 4.6 would imply that
for our polynomial ring $R$ there is a class ${\mathcal D}$ of
pairs $(R, M)\in \D^d$ which is
not of bounded cohomology but such that the set $\{d^i_M(n)\mid \ (R, M)\in \D\}$
is finite for all $(i, n)\in \mathbb{S}$. As $\C$ is a big class, we have ${\mathcal D}
\subseteq {\mathcal C}$, and this would imply the contradiction that
${\mathbb S}$ does not bound cohomology in $\C$.
\end{proof}

\begin{thm}\label{7.9 Theorem}%-------------------------------------------------------------------------------------------------------------------------7.8

Let ${\mathcal C} \subseteq \D^d$ be a big class and let ${\mathbb
S} \subseteq \{ 0, \cdots , d-1 \} \times {\mathbb Z}$. Then
${\mathbb S}$ bounds cohomology in ${\mathcal C}$ if and only if
${\mathbb S}$ contains a quasi-diagonal.
\end{thm}

\begin{proof} Clear by Corollary 4.4 and
Proposition 4.8.
\end{proof}

\begin{cor} \label{7.10 Corollary}%-------------------------------------------------------------------------------------------------------------------------7.9

The set ${\mathbb S} \subseteq \{ 0, \cdots , d-1\} \times
{\mathbb Z}$ bounds cohomology in $\D^d$ if and only if ${\mathbb
S}$ contains a quasi-diagonal.
\end{cor}

\begin{proof} Clear by Theorem 4.9.
\end{proof}

\vskip 1 cm

%%%%%%%%%%%%%%%%%%%%%%%%
\section{Bounding Invariants}%------------------------------------------------------------------------------------------------------------------------sec.8
%%%%%%%%%%%%%%%%%%%%%%%%
\label{8. Bounding Invariants}

{\rm In this section we investigate numerical invariants which bound
cohomology.}

\smallskip

\begin{defins}\label{8.1 Definition}

\rm{(A) (s. [2], [8], [9]). Let ${\mathcal C} \subseteq {\mathcal
D}^d$ be a subclass. A {\it numerical invariant} \rm{on the class
${\mathcal C}$ is a map
\[ \mu : {\mathcal C} \rightarrow {\mathbb Z} \cup \{ \pm \infty \} \]

such that for any two pairs $(R, M), (R, N) \in {\mathcal C}$ with
$M \cong N$ we have }}$\mu (R, M) = \mu (R, N)$. We shall write
$\mu(M)$ instead of $\mu(R, M)$.

\smallskip

\rm{(B) Let $(\mu _i)^r_{i = 1}$ be a family of numerical
invariants on the subclass ${\mathcal C} \subseteq {\mathcal
D}^d$. We say that the family $(\mu_i)_{i=1}^r$
{\it bounds cohomology on the class} ${\mathcal C}$, if for each $(n_1, \cdots , n_r)
\in ({\mathbb Z} \cup \{ \pm \infty \} )^r$ the class
\[ \{ (R, M) \in {\mathcal C} \big\arrowvert \mu _i(M) = n_i \mbox{ for all } i \in \{ 1,\cdots , r
\}\}\] is of bounded cohomology.}

\smallskip

\rm{(C) A numerical invariant $\mu $ on the class ${\mathcal C}
\subseteq {\mathcal D}^d$ is said to be {\it finite} if $\mu (M) \in
{\mathbb Z}$ for all} $(R, M) \in {\mathcal C}$.

\smallskip

\rm{(D) A numerical invariant $\mu $ on the class ${\mathcal C}
\subseteq {\mathcal D}^d$ is said to be {\it positive} if $\mu (M)
\geq 0$ for all }$(R, M) \in {\mathcal C}$. \hfill $\bullet$
\end{defins}

\begin{rem}\label{8.2 Remark}%__________-------------

\rm{(A) If $\mu : {\mathcal C} \rightarrow {\mathbb Z} \cup \{ \pm
\infty \} $ is a numerical invariant on the class ${\mathcal C}
\subseteq {\mathcal D}^d$ and if ${\mathcal D} \subseteq {\mathcal
C}$, then the restriction $\mu \upharpoonright _{\mathcal D} :
{\mathcal D} \rightarrow {\mathbb Z} \cup \{ \pm \infty \} $ is a
numerical invariant on the class ${\mathcal D}$. Clearly, if $\mu $
is finite (resp. positive) then so is} $\mu \upharpoonright
_{\mathcal D}$.

\smallskip

\rm{(B) If $(\mu _i)^r_{i = 1}$  bounds cohomology on the class ${\mathcal C} \subseteq {\mathcal
D}^d$ and if ${\mathcal D} \subseteq {\mathcal C}$, then $(\mu _i
\upharpoonright _{\mathcal D})^r_{i = 1}$  bounds cohomology in }${\mathcal D}$.

\smallskip

\rm{(C) A family $(\mu _i)^r_{i = 1}$ of positive numerical
invariants bounds cohomology in ${\mathcal C}$ if and only if for
all $(n_1, \cdots , n_r) \in ({\mathbb N}_0 \cup \{ \infty \})^r $
the class
\[ \{ (R, M) \in {\mathcal C} \big\arrowvert \mu _i
      (M) \leq n_i \mbox{ for all } i \in \{ 1, \cdots ,
      r \} \}
   \]
is of bounded cohomology.}

\smallskip

\rm{(D) A family $(\mu _i)^r_{i = 1}$ of finite positive invariants
bounds cohomology on ${\mathcal C}$ if and only if the sum invariant
$\sum ^r _{i = 1} \mu _i : {\mathcal C} \rightarrow {\mathbb N}_0$
bounds cohomology in ${\mathcal C}$}. \hfill $\bullet$
\end{rem}

\begin{rem}\label{8.3 Notation}%------------------------------------------------------------

\rm{Let $i \in {\mathbb N}_0$ and $n \in {\mathbb Z}$. Then, the map
\[d^i_{\bullet}(n) : {\mathcal D}^d \rightarrow {\mathbb N}_0 ; \ ((R, M)
\mapsto d^i_{M}(n)) \]

is a finite positive numerical invariant on ${\mathcal D}^d$.}  \hfill $\bullet$
\end{rem}

\begin{thm}\label{8.4 Theorem}%----------------------------------------
 Let $(n_i)^{d-1}_{i = 0}$ be a sequence of integers
such that $n_0 > n_1 > n_2 > \ldots > n_{d-1}$. Then the family of numerical invariants
$(d^i_{\bullet}(n_i)) ^{d-1}_{i = 0}$  bounds   cohomology in ${\mathcal D}^d$.
\end{thm}

\begin{proof}
Clear by Proposition 4.2.
\end{proof}

\begin{remi}\label{8.5 Reminder}%---------------------------------------

\rm{For each $k\in {\N}_0$ we may define the numerical invariant}
\[\mbox{reg}^k : {\mathcal D}^d \rightarrow {\mathbb Z} \cup \{ -
\infty \} ; \ ((R, M) \mapsto \mbox{reg}^k(M)).\]

\hspace*{\fill $\bullet $}
\end{remi}

\begin{nota}\label{8.6 Notation}%---------------------------------------

\rm{For $(R, M) \in {\mathcal D}^d$ we set}
\[ \varrho (M):= \begin{cases}
d^0_{M}(\mbox{reg}^2(M)), &\mbox{if } \dim (M)
> 1 , \\ d^0_{M}(0), &\mbox{if } \dim (M)
\leq 1. \end{cases}\]  \hfill $\bullet$
\end{nota}

\begin{rem}\label{8.7 Remark}%---------------------------------------
\rm{(A) If $(R, M) \in {\mathcal D}^d$ with $\dim (M) \leq 1$, the
cohomological Hilbert function $d^0_{M}$ of $M$
is constant, and this constant is strictly positive if and only if
$M \not= 0$.}

\smallskip

\rm{(B) The function
\[ \varrho : {\mathcal D}^d \rightarrow {\mathbb N}_0 ; \
((R, M) \mapsto \varrho (M))\]

is a finite positive numerical invariant on }${\mathcal D}^d$.
\hfill $\bullet $
\end{rem}

\begin{thm}\label{8.8 Theorem}%---------------------------------------

The pair of invariants $(\reg^2, \varrho )$   bounds
cohomology in ${\mathcal D}^d$.
\end{thm}

\begin{proof} Fix $u, v \in {\mathbb Z}$ and set
\[ {\mathcal C}:= \{ (R, M) \in {\mathcal D}^d
\big\arrowvert \mbox{reg}^2(M) = u, \varrho (M) = v \}.\]

If $(R, M) \in {\mathcal C}$ we have $d^0_{M}(u) =
d^0_{M}(\mbox{reg} ^2(M)) = v$.

\smallskip

Let $i \in {\mathbb N}$. Then $u - i = \mbox{reg}^2(M) - i >a
_{i+1}(M)$ and hence $d^i_{M}(u - i) = h^{i+1}_M(u - i) = 0$.
Therefore $(R, M)$ belongs to the class
\[ {\mathcal D}:= \{ (R, M) \in {\mathcal D}^d \big
\arrowvert d^0_M(u) = v \mbox{ and } d^i_{M} (u - i) = 0 \mbox{
for all } i \in \{ 1, \cdots ,d-1 \} \}.\]

But according to Theorem 5.4 the class ${\mathcal D}$ is of bounded
cohomology.
\end{proof}

\smallskip

\begin{lem}\label{8.9 Lemma}%---------------------------------------
Let $(R, M)\in \D^d$ be such that $\dim(R/\p)\neq 1$ for all
$\p\in \Ass_R(M)$. Then
\[d^0_M(n-1)\leq \max\{0, d^0_M(n)-1\} \ \mbox{for all} \  n\in \Z.\]
\end{lem}

\begin{proof}
For an arbitrary finitely generated graded $R$-module $N$ let
\[\lambda(N):= \inf\{\depth(N_{\p})+ \height ((\p+R_+)/\p)\mid \ \p\in \Spec (R)\backslash \Var (R_+)\}.\]
Clearly, for all $n\in \Z$ we have $\lambda(N(n))= \lambda(N)$.
So, for all $n\in \Z$, we get by our hypotheses that
$\lambda(M(n))= \lambda(M)>1$. Now, according to [8, Proposition
4.6] we obtain
\[d^0_M(n-1)= d^0_{M(n)}(-1)\leq \max\{0,  d^0_{M(n)}(0)-1\}= \max\{0,  d^0_{M}(n)-1\}.\]
\end{proof}

\smallskip

\begin{thm}\label{8.10 Theorem}%---------------------------------------

Let $r,s \in {\mathbb Z}$ and let $p \in {\mathbb Q}[t]$ be a
polynomial. Let ${\mathcal C} \subseteq {\mathcal D}^d$ be the class
of all pairs $(R, M) \in {\mathcal D}^d$ satisfying the following
conditions:

\smallskip

\renewcommand{\descriptionlabel}[1]%
             {\hspace{\labelsep}\textrm{#1}}
\begin{description}
\setlength{\labelwidth}{13mm} \setlength{\labelsep}{1.3mm}
\setlength{\itemindent}{0mm}

\item[{$(\alpha )$}] There is a finitely generated graded R-module $N$ with
Hilbert polynomial
$ p_{N} = p$ and  $\reg^2
(N) \leq r $ such that $M \subseteq N.$

\smallskip

\item[$(\beta )$] ${\rm reg}^2(M) \leq s$.
\end{description}

\smallskip

Then, ${\mathcal C}$ is a class of finite cohomology.
\end{thm}

\begin{proof} Let $v:= \max \{ r, s\} $. We first show that for each
pair $(R, M) \in {\mathcal C}$ we have
\[  \ \varrho(M)\leq p(v) \leqno{(\ast )}\]
 and
\[  \ \dim(M)\leq 1 \ \mbox{or} \ {\rm reg}^2(M) \geq  -v - p(v).\leqno{ (\ast \ast )}\]

So, let $(R, M) \in {\mathcal C}$. Then, there is a monomorphism
of finitely generated graded $R$-modules $M
\stackrel{\epsilon}{\rightarrowtail} N$ such that $p_{N} = p$ and
${\rm reg}^2(N) \leq r\leq  v$.

\smallskip

Assume first that $\dim(M)>1$. As ${\rm reg}^2(M) \leq v$ we then
get
\[ \varrho(M)= d^0_M({\rm reg}^2(M)) \leq d^0_{M} (v) \leq
d^0_{N} (v) = p_{N} (v) = p(v).\] If $\dim(M)\leq 1$, the
function $d^0_M$ is constant and therefore
\[ \varrho(M)=
d^0_M(0)= d^0_{M} (v)\leq d^0_{N} (v)= p_N(v)= p(v).\] Thus we
have proved statement $(\ast )$.

\smallskip

To prove statement $(\ast \ast)$ we assume that $\dim(M)>1$. Then
there is a short exact sequence of finitely generated graded $R$-
modules
\[0\longrightarrow H\longrightarrow M\longrightarrow \overline{M}\longrightarrow 0\]
such that $\dim(H)\leq 1$ and $\mbox{Ass}_R(\overline{M})$ does not
contain any prime $\p$ with $\dim(R/\p)\leq 1$. As $\dim(H)\leq
1$, we have $H^i_{R_+}(H)=0$ for all $i>1$. Therefore
$H^i_{R_+}(M)\cong H^i_{R_+}(\overline{M})$ for all $i>1$ and hence
$\reg ^2 (M)= \reg ^2 (\overline{M})$. Moreover by the observation
made on $\mbox{Ass}_R(\overline{M})$, we have (s. Lemma 5.9)
\[d^0_{\overline{M}}(n-1)\leq \max \{0, d^0_{\overline{M}}(n)-1\} \ \mbox{for all} \
n \in \Z.\]
As $D^1_{R_+}(H)= H^2_{R_+}(H)=0$, we have
\[d^0_{\overline{M}}(v)\leq d^0_M(v)\leq d^0_N(v)= p_N(v)= p(v)\]
and it
follows that
\[d^0_{\overline{M}}(n)=0 \ \mbox{for all} \ n\leq -v-p(v)-1.\]
One consequence of this is, that $T:= D^0_{R_+}(\overline{M})$ is a
finitely generated $R$-module. As $H^i_{R_+}(M)\cong
H^i_{R_+}(\overline{M}) $ for all $i>1$, we have $\reg ^2(T)= \reg
^2(\overline{M})= \reg ^2(M)$. As $H^i_{R_+}(T)=0$ for $i= 0, 1$, we thus
get $ \reg ^2(M)=  \reg (T)$. As $T_n= 0$ for all $n\leq -v-p(v)-1$,
we finally obtain (s. Reminder 2.2(E))
\[\reg ^2(M)=  \reg
(T)\geq \gendeg(T)\geq \beg(T)\geq -v-p(v).\] This proves
statement $(\ast \ast )$.

\smallskip

Now, we may write
\[ {\mathcal C} \subseteq {\mathcal C}_{-\infty} \cup \bigcup ^s_{t =
 -v - p(v)} {\mathcal C}_t ,\]
where
\[{\mathcal C}_{-\infty}:=  \{ (R, M) \in {\mathcal D}^d
\big\arrowvert \dim(M)\leq 1 \ \mbox{and} \  \varrho(M)  \leq
p(v)\}\]
and, for all $t \in {\mathbb Z}$ with $-v - p(v)\leq t \leq s$,
\[ {\mathcal C}_t:= \{ (R, M) \in {\mathcal D}^d \big\arrowvert
{\rm reg}^2(M) = t, \ \varrho(M) \leq p(v)\}.\]

The class ${\mathcal C}_{-\infty}$ clearly is of bounded
cohomology.

Now, by Remark 5.2(C) and by Corollary 5.8, each of the classes
${\mathcal C}_t$ is of bounded cohomology. This proves our claim.
\end{proof}

\begin{cor}\label{8.10 Corollary}%---------------------------------------

\smallskip

Let $r \in {\mathbb Z}$ and let $p \in {\mathbb Q}[t]$ be a
polynomial. Let ${\mathcal C} \subseteq {\mathcal D}^d$ be the
class of all pairs $(R, M) \in {\mathcal D}^d$ satisfying the
condition $(\alpha )$ of Theorem 5.10. Then, the invariant ${\rm
reg}^2$ bounds cohomology in the class ${\mathcal C}$.
\end{cor}

\begin{proof}
This is immediate by Theorem 5.10.
\end{proof}

\begin{cor}\label{8.11 Corollary}%---------------------------------------

Let $r \in {\mathbb Z}$ and let $(R, N) \in {\mathcal D}^d$. If $M$ runs through all graded
 submodules $M \subseteq N$ with
${\rm reg}^2 (M) \leq r$, only finitely many cohomology tables
$d_{M}$ and hence only finitely many Hilbert polynomials $p _{M}$
occur.
\end{cor}
\begin{proof} This is clear by Theorem 5.10.
\end{proof}

\smallskip

\begin{cor}\label{8.12 Corollary}%---------------------------------------
Let $r\in \Z$ and let $(R, N)\in \D^d$. If $M$ runs through all
graded submodules of $N$ with $\reg^1(M)\leq r$ only finitely many
families
\[(h^i_M(n))_{(i, n)\in {\N}_0\times \Z} \ \ \
\mbox{and} \ \ \ (h^i_{N/M}(n))_{(i, n)\in
{\N}_0\times \Z}\] can occur.
\end{cor}
\begin{proof}
Let $\X$ be the set of all graded submodules $M\subseteq N$ with
$\reg^1(M)\leq r$.

\smallskip

Now, for each $M\in \X$ we have the following three relations
\[d^i_M(n)= h^{i+1}_M(n) \ \mbox{for all} \ i\geq 1 \ \mbox{and all} \ n\in \Z;\]
\[\begin{cases}h^1_M(n)\leq d^0_M(n) \ &\mbox{for all} \ n\in \Z;\\
h^1_M(n)= d^0_M(n) \ &\mbox{for all} \ n< \beg(N);\\ \
h^1_M(n)= 0  \ &\mbox{for all} \ n\geq r\end{cases}\]
and
\[h^0_M(n)\leq h^0_N(n) \ \mbox{for all} \ n\in \Z.\]
So, by Corollary 5.12 the set
\[\mathcal{U}:=\{(h^i_M(n))_{(i, n)\in {\N}_0\times \Z}\mid \ M\in \X\}\]
is finite.

\smallskip

For each $M\in \X$ the short exact sequence
$0\longrightarrow M\longrightarrow N\longrightarrow
N/M\longrightarrow 0$ yields that for all $n\in \Z$ and all $i\in
{\N}_0 $
\[\hskip5.7cm h^0_{N/M}(n)\leq h^0_{N}(n)+ h^1_{M}(n),\hskip5.6cm(1)\]
\[\hskip5.7cm d^i_{N/M}(n)\leq d^i_{N}(n)+ h^{i+2}_{M}(n). \hskip5.6cm(2)\]
By the finiteness of $\mathcal{U}$ it follows that the set of
functions
\[\mathcal{U}_0:=\{(h^0_{N/M}(n))_{n\in \Z}\mid \ M\in \X\}\]
is finite and that the set of cohomology diagonals
\[\mathcal{W}:=\{(d^i_{N/M}(-i))_{i=0}^{d-1}\mid \ M\in \X \}\]
is finite.

\smallskip

In view of Theorem [6, Theorem 5.4] the finiteness of
$\mathcal{W}$ implies that the set
\[\mathcal{U}_1:=\{(d^i_{N/M}(n))_{(i, n)\in {\N}_0\times \Z}\mid \ M\in \X\}\]
is finite. Moreover for all $M\in \X$ we have
\[\en(H^1_{R_+}(N/M))< \reg^1(N/M)\leq \max\{\reg^2(M)-1, \reg^2(N)\}\leq \max\{r-1, \reg^1(N)\};\]
\[ \ h^1_{N/M}(n)\leq d^0_{N/M}(n) \ \mbox{for all} \ n\in \Z, \ \mbox{with equality if} \ n< \beg(N).\]
As $d^i_{N/M}\equiv h^{i+1}_{N/M}$ for all $i>0$ the finiteness of
$\mathcal{U}_0$ and $\mathcal{U}_1$ shows that the set
\[\{(h^i_{N/M}(n))_{(i, n)\in
{\N}_0\times \Z}\mid \ M\in \X\}\]
 is finite, too.
\end{proof}

\smallskip

\begin{cor}\label{8.14 Corollary}%---------------------------------------
Assume that $R$ is a homogeneous Noetherian Cohen-Macaulay ring with Artinian local
base ring $R_0$. Let $s\in \Z$ and let $N$ be a finitely
generated graded $R$-module. If $M$ runs trough all
graded submodules of $N$ with $\gendeg(M)\leq s$ only finitely many
families
\[(h^i_{M}(n))_{(i, n)\in {\N}_0\times \Z} \ \ \  \mbox{and} \ \ \  (h^i_{N/M}(n))_{(i, n)\in {\N}_0\times \Z}\]
may occur.
\end{cor}
\begin{proof}
By [4, Proposition 6.1] we see that $\reg(M)$ finds an upper
bound in terms of $\gendeg(M)$, $\reg(N)$, $\reg(R)$, $\beg(N)$,
$\dim(R)$, the multiplicity $e_0(R)$ of $R$ and the minimal number
of homogeneous generators of the $R$-module $N$. Now, we conclude
by Corollary 5.13.
\end{proof}

\smallskip

\begin{rem}\label{8.15 Remark}%---------------------------------------
\rm{If we apply Corollary 5.13 in the special case where $N= R=
K[x_1, \cdots, x_r]$ is a polynomial ring over a field, we get
back the finiteness result [17, Corollary 14]. Correspondingly,
if we apply Corollary 5.14 in this special case, we get back [17,
Corollary 20].}
\end{rem}

\vskip 1 cm

{\bf Acknowledgment.}   {\rm The third author would like to thank the
Institute of Mathematics of University of Z\"urich for financial
support and hospitality during the preparation of this paper.}

\renewcommand{\descriptionlabel}[1]%
             {\hspace{\labelsep}\textrm{#1}}
\begin{description}
\setlength{\labelwidth}{12mm} \setlength{\labelsep}{1.3mm}
\setlength{\itemindent}{0mm} {\rm

\item [{[1]}] D. Bayer and D. Mumford, \emph{What can be computed in
algebraic geometry?} in "Computational Algebraic Geometry and
Commutative Algera" Proc. Cortona 1991 (D.Eisenbud and L.
Robbiano, Eds.) Cambridge University Press (1993) 1-48.

\item [{[2]}]  M. Brodmann,
\emph{Cohomological invariants of coherent sheaves over
projective schemes - a survey}, in ''Local Cohomology and its
Applications'' (G. Lyubeznik, Ed), 91 - 120, M. Dekker Lecture
Notes in Pure and Applied Mathematics {\bf226} (2001).

\item [{[3]}] M. Brodmann, \emph{Castelnuovo-Mumford regularity and degrees of generators of graded submodules},
 Illinois J. Math. {\bf47}, no. 3 (Fall 2003).

\item [{[4]}] M. Brodmann and T. G\"otsch, \emph{Bounds for the
Castelnuovo-Mumford regularity}, to appear in Journal of
Commutative Algebra.

\item [{[5]}]
M. Brodmann, M. Hellus, {\em Cohomological patterns of coherent
sheaves over projective schemes}, Journal of Pure and Applied
Algebra {\bf 172} (2002), 165-182.

\item [{[6]}] M. Brodmann, M. Jahangiri and C. H. Linh, \emph{Castelnuovo-Mumford regularity of deficiency modules},
 Preprint 2009.

\item [{[7]}]
M. Brodmann, F. A. Lashgari, {\em A diagonal bound for
cohomological postulation numbers of projective schemes}, J.
Algebra {\bf 265} (2003), 631-650.

\item [{[8]}]
M. Brodmann, C. Matteotti and N. D. Minh, {\em Bounds for
cohomological Hibert functions of   projective schemes over Artinian
rings}, Vietnam Journal of Mathematics {\bf 28}(4)(2000), 345-384.

\item [{[9]}]
M. Brodmann, C. Matteotti and N. D. Minh, {\em Bounds for
cohomological deficiency functions  of   projective schemes over
Artinian rings}, Vietnam Journal of Mathematics {\bf 31}(1)(2003),
71-113.

\item [{[10]}]
M. Brodmann and R.Y. Sharp, {\em Local cohomology: an algebraic
introduction with geometric applications}, Cambridge University
Press (1998).

\item [{[11]}]
G. Caviglia, {\em Bounds on the Castelnuovo-Mumford regularity of
tensor products}, Proc. AMS. {\bf135} (2007) 1949- 1957.

\item [{[12]}]
G. Caviglia and E. Sbarra, \emph{Characteristic free bounds for
the Castelnuovo-Mumford regularity}, Compos. Math. {\bf141} (2005)
1365- 1373.

\item [{[13]}]
M. Chardin, A.L. Fall and U. Nagel, \emph{Bounds for the
Castelnuovo-Mumford regularity of modules}, Math. Z. {\bf258}
(2008) 69-80.

\item [{[14]}]
S. Fumasoli, \emph{Die K\"unnethrelation in abelschen Kategorien
und ihre Anwendung auf die Idealtransformation}, Diplomarbeit,
Universit\"at Z\"urich, 1999.

\item [{[15]}]
A. Grothendieck, \emph{S\'eminare de g\'eometrie alg\'ebrique IV},
Springer Lecture Notes 225, Springer (1971).

\item [{[16]}] L.T. Hoa,
\emph{Finiteness of Hilbert functions and bounds for the
Castelnuovo-Mumford regularity of initial ideals}, Trans. AMS.
{\bf360} (2008) 4519- 4540.

\item [{[17]}]
L. T. Hoa and E. Hyry {\em Castelnuovo-Mumford regularity of
canonical and deficiency modules},
J. Algebra {\bf 305} (2006) no.2, 877-900.

\item [{[18]}]
J. Kleiman  , \emph{Towards a numerical theory of ampleness}, Annals
of Math. {\bf84} (1966) 293- 344.

\item [{[19]}]
C. H. Linh, {\em   Upper bounds for Castelnuovo-Mumford regularity of
associated graded modules},  Comm. Algebra, {\bf 33}(6) (2005),
1817-1831.

\item [{[20]}]
S. Mac Lane, {\em   Homology},  Die Grundlehren der mathematischen Wissenschaften 114, Springer, Berlin (1963).

\item [{[21]}]
D. Mumford, \emph{Lectures on curves on an algebraic surface},
Annals of Math. Studies {\bf59}, Princeton University Press (1966).

\item [{[22]}]
M.E. Rossi, N.V. Trung and G. Valla, {\em Castelnuovo-Mumford
regularity and extended degree}, Trans. Amer. Math. Soc. {\bf 355}
(2003), no. 5, 1773-1786.

\item [{[23]}]
J. St\"ukrad and W. Vogel, \emph{Buchsbaum Rings and
Applications}, Deutscher Verlag der Wissenschaften, Berlin
(1986).}

\end{description}
%\end{thebibliography}
\end{document}